\documentclass[12pt]{article}
\usepackage{amsmath,amssymb,amsthm}
\usepackage{graphicx}
\usepackage{subfigure}
\usepackage{amsmath,amsfonts,amssymb,amscd}
\usepackage{indentfirst,graphicx,epstopdf}
\usepackage{graphicx,psfrag,ifpdf,enumerate}
\usepackage{caption}
\usepackage{amsthm, amsfonts, color}
\usepackage[bookmarksnumbered, plainpages]{hyperref}
\usepackage[numbers,sort&compress]{natbib}
\usepackage[noend]{algpseudocode}
\usepackage{algorithmicx,algorithm}
\usepackage{psfrag}
\usepackage{color}
\usepackage{enumerate}
\usepackage{epstopdf}
\input{epsf}

\newcommand\tabcaption{\def\@captype{table}\caption}

\newtheorem{theorem}{Theorem}

\newtheorem{cons}[theorem]{Construction}
\newtheorem{lemma}[theorem]{Lemma}
\newtheorem{claim}{Claim}[section]

\theoremstyle{definition}

\newtheorem{fact}[theorem]{Fact}

\title{\Large \bf Rainbow triangles in edge-colored complete graphs}
\date{}
\author{{\small Xiaozheng Chen, Xueliang Li}\\
{\small Center for Combinatorics and LPMC}\\
{\small Nankai University, Tianjin 300071, China}\\
{\small Email: cxz@mail.nankai.edu.cn, lxl@nankai.edu.cn}
}

\begin{document}

\maketitle

\begin{abstract}
Let $G$ be a graph of order $n$ with an edge-coloring $c$,
and let $\delta^c(G)$ denote the minimum color-degree of $G$.
A subgraph $F$ of $G$ is called rainbow if any two edges of $F$
have distinct colors. There have been a lot results in the existing literature on
rainbow triangles in edge-colored complete graphs. Fujita and Magnant showed that
for an edge-colored complete graph $G$ of order $n$, if $\delta^c(G)\geq \frac{n+1}{2}$, then
every vertex of $G$ is contained in a rainbow triangle.
In this paper, we show that if $\delta^c(G)\geq \frac{n+k}{2}$,
then every vertex of $G$ is contained in at least $k$ rainbow triangles, which can be seen as a
generalization of their result.
Li showed that for an edge-colored graph $G$ of order $n$,
if $\delta^c(G)\geq \frac{n+1}{2}$, then $G$ contains a rainbow triangle.
We show that if $G$ is complete and $\delta^c(G)\geq \frac{n}{2}$,
then $G$ contains a rainbow triangle and the bound is sharp.
Hu et al. showed that for an edge-colored graph $G$ of order $n\geq 20$,
if $\delta^c(G)\geq \frac{n+2}{2}$, then $G$ contains two vertex-disjoint rainbow triangles.
We show that if $G$ is complete with order $n\geq 8$ and $\delta^c(G)\geq \frac{n+1}{2}$,
then $G$ contains two vertex-disjoint rainbow triangles.
Moreover, we improve the result of Hu et al. from $n\geq 20$ to $n\geq 7$, the best possible.\\[2mm]
{\bf AMS Classification 2020:} 05C15, 05C38.

\end{abstract}

\section{Introduction}

In this paper, we consider finite simple undirected graphs.
An edge-coloring of a graph $G$ is a mapping $c: E(G)\rightarrow \mathbb{N}$,
where $\mathbb{N}$ denotes the set of natural numbers.
A graph $G$ is called an \emph{edge-colored graph}
if $G$ is assigned an edge-coloring.
The color of an edge $e$ of $G$ and the set of colors assigned to $E(G)$
are denoted by $c(e)$ and $C(G)$, respectively. For subset $X$
of vertices of $G$, we use $G[X]$ to denote the subgraph of $G$ induced by
$X$. For $V_1,V_2\subset V(G)$ and $V_1\cap V_2=\emptyset$,
we set $E(V_1,V_2)=\{xy\in E(G),x\in V_1,y\in V_2\}$,
and when $V_1=\{v\}$, we write $E(u,V_2)$ for $E(\{u\},V_2)$.
The set of colors appearing on the edges between $V_1$ and $V_2$ in $G$
is denoted by $C(V_1,V_2)$.
When $V_1=\{v\}$, use $C(v,V_2)$ instead of $C(\{v\},V_2)$.
The set of colors appearing on the edges of a subgraph $H$ of $G$,
is denoted by $C(H)$; moreover if $H=G[V_1]$,
we write $C(V_1)$ for $C(G[V_1])$.
A subset $F$ of edges of $G$ is called \emph{rainbow}
if no pair of edges in $F$ receive the same color,
and a graph is called \emph{rainbow} if its edge-set is rainbow.
In this paper, we only consider rainbow triangles in an edge-colored complete graph.

For a vertex $v\in V(G)$, the \emph{color-degree}
of $v$ in $G$ is the number of distinct colors assigned to the edges incident to $v$,
denoted by $d^c_G(v)$.
We use $\delta^c(G)=\mbox{min}\{d^c_G(v):v\in V(G)\}$
to denote the \emph{minimum color-degree} of $G$.
The set of neighbors of a vertex $v$
in a graph $G$ is denoted by $N_G(v)$.
Let $N_i(v)$ denote the set of vertices with edges of
color $i$ adjacent to $v$ for $1\leq i\leq d^c(v)$,
that is $N_i(v)=\{u\in N_G(v),c(uv)=i\}$.
Let $\Delta^{mon}(v)$ be the maximum number of incident edges of $v$
with the same color,
that is $\Delta^{mon}(v)=\mbox{max}\{|N_i(v)|,1\leq i\leq d^c(v)\}$.
Then the \emph{monochromatic-degree} of $G$ is the maximum
$\Delta^{mon}(v)$ over all vertices $v$ in $G$, denoted by $\Delta^{mon}(G)$.
Let $\psi$ be the incidence function
that associates with each edge of $G$ an unordered pair of (not necessarily
distinct) vertices of $G$.
If $e$ is an edge and $u$ and $v$ are vertices such that $\psi (e) =
\{u,v\}$, then $e$ is said to join $u$ and $v$, and the vertices $u$ and $v$ are called the ends
of $e$. Let $R$ be a subset of $E(G)$.
Then $\psi(R)$ denotes the set of all vertices incident with the edges in $R$,
that is $\psi(R)=\cup _{e\in R}\psi(e)$.
For other notation and terminology not defined here, we refer to \cite{B}.

There have been many results on rainbow triangles in the existing literature.
These results can be divided into two parts:
local property and global property on rainbow triangles.
As for local property, Fujita and Magnant showed the following result.

\begin{theorem}[\cite{FM}]\label{1}
Let $G$ be an edge-colored complete graph of order $n$.
If $\delta^c(G)\geq \frac{n+1}{2}$,
then every vertex of $G$ is contained in a rainbow triangle.
\end{theorem}

The lower bound on $\delta^c(G)$ in Theorem \ref{1} is sharp.
To see this, we obtain the following construction.
It will show that $v$ is not contained in any rainbow triangle.

\begin{cons}\label{2}
Consider a complete graph $G=K_{2n}$.
Let $v$ be a vertex of $G$ such that $d^c(v)=n$.
Set $|N_1(v)|=1$ and $|N_i(v)|=2$ for $2\leq i\leq n$.
Color the edges between $N_1(v)$ and $N_i(v)$
by $i$ for $2\leq i\leq n$.
For any vertex $u\in N_i(v)$,
color two edges between $u$ and $N_j(v)$ by $i$ and $j$, respectively,
for $2\leq i\neq j\leq n$.
Color the edge in $G[N_i(v)]$ by a new color different from the colors of edges incident with $v$.
Then we get an edge-colored complete graph with
$\delta^c(G)=n$; see Figure \ref{fig1}.
\end{cons}

\begin{figure}[htbp]
  \centering
 \scalebox{1}{\includegraphics[width=2.4in,height=1.5in]{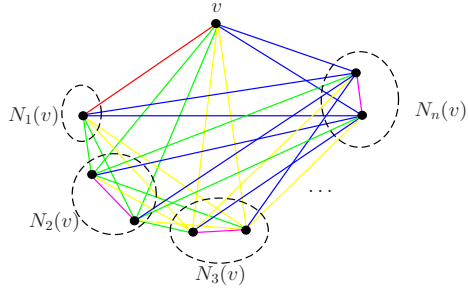}}\\
 \captionsetup{font={scriptsize}}
  \caption{The structure of $G$ in Construction \ref{2}}\label{fig1}
\end{figure}

Using Theorem \ref{1}, and repeatedly deleting the vertices
of rainbow triangles at $v$,
it is easy to obtain the following sufficient condition for the
existence of $k$ edge-disjoint rainbow triangles at $v$
and the lower bound is sharp.

\begin{fact}\label{4}
Let $G$ be an edge-colored complete graph of order $n$.
If $\delta^c(G)\geq \frac{n-1}{2}+k$,
then every vertex of $G$ is contained in at least $k$ edge-disjoint rainbow triangles.
\end{fact}

In this paper, we will show the following result,
which can be seen as a generalization of Theorem \ref{1}.

\begin{theorem}\label{3}
Let $G$ be an edge-colored complete graph of order $n$.
If $\delta^c(G)\geq \frac{n+k}{2}$,
then every vertex of $G$ is contained in at least $k$ rainbow triangles.
\end{theorem}

As for global property on rainbow triangle,
there are some results in an edge-colored general graph.

\begin{theorem}[\cite{L}] \label{15}
Let $G$ be an edge-colored graph of order $n$.
If $\delta^c(G)\geq \frac{n+1}{2}$,
then $G$ contains a rainbow triangle.
\end{theorem}

\begin{theorem}[\cite{LNXZ}]\label{16}
Let $G$ be an edge-colored graph of order $n$.
If $\delta^c(G)\geq \frac{n}{2}$ and $G$ contains no rainbow triangles,
then $n$ is even and $G$ is the complete bipartite graph $K_{\frac{n}{2},\frac{n}{2}}$, unless $G=K_4-e$ or $K_4$  when $n=4$.
\end{theorem}

Recently, Hu et al. proved the following result in an edge-colored general graph.

\begin{theorem}[\cite{HLY}]\label{12}
Let $G$ be an edge-colored graph of order $n\geq 20$.
If $\delta^c(G)\geq \frac{n+2}{2}$,
then $G$ contains two vertex-disjoint rainbow triangles.
\end{theorem}

In this paper, we are seeking for sufficient condition
for the existence of rainbow triangles in an edge-colored complete graph.

\begin{theorem}\label{8}
Let $G$ be an edge-colored complete graph of order $n$.
If $\delta^c(G)\geq \frac{n}{2}$,
then $G$ contains a rainbow triangle.
\end{theorem}

With more effort, we can obtain the following stronger theorem.

\begin{theorem}\label{10}
Let $G$ be an edge-colored complete graph of order $n$.
If $\delta^c(G)\geq \frac{n-1}{2}$ and $G$ contains no rainbow triangle,
then $V(G)$ can be partitioned into $\frac{n+1}{2}$ parts $\{A_0,A_1,\cdots, A_{\frac{n-1}{2}}\}$ (see Figure \ref{fig2}),
such that the following properties hold:

(1) $n$ is odd and $d^c(v)=\frac{n-1}{2}$ for all $v\in V(G)$;

(2) $|A_0|=1$ and $|A_i|=2$ for $1\leq i\leq \frac{n-1}{2}$;

(3) for any vertex $u\in A_i$, $C(u,A_i)= \{i,j\}$ for $1\leq i\neq j\leq \frac{n-1}{2}$;

(4) if $\frac{n-1}{2}\leq 2$, then $C(A_i)\in \{1,2\}$, $i=1,2$;
if $\frac{n-1}{2}\geq 3$, then $C(A_i)=\{i\}$, $1\leq i\leq \frac{n-1}{2}$.
\end{theorem}

\begin{figure}[htbp]
  \centering
 \scalebox{1}{\includegraphics[width=2.2in,height=1.8in]{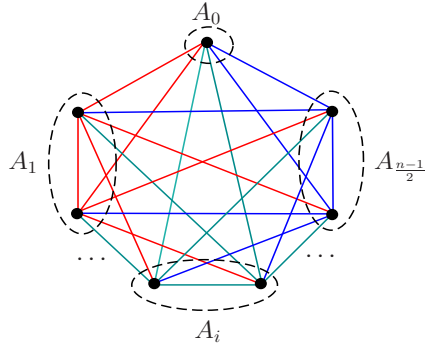}}\\
 \captionsetup{font={scriptsize}}
  \caption{The structure of $G$ in Theorem \ref{10}}\label{fig2}
\end{figure}

Theorem \ref{10} shows that the lower bound on $\delta^c(G)$ in Theorem \ref{8}
is tight.
Using Theorem \ref{8}, and repeatedly deleting the vertices
of rainbow triangles
it is easy to obtain the following sufficient condition for
the existence of $k$ vertex-disjoint rainbow triangles.

\begin{fact}\label{13}
Let $G$ be an edge-colored complete graph of order $n$.
If $\delta^c(G)\geq \frac{n-3+3k}{2}$,
then $G$ has at least $k$ vertex-disjoint rainbow triangles.
\end{fact}

The lower bound on $\delta^c(G)$ is far from tight.
We will investigate the minimum color-degree condition that
guarantees the existence of two vertex-disjoint rainbow triangles in an edge-colored complete graph.

\begin{theorem}\label{11}
Let $G$ be an edge-colored complete graph of order $n$.
If $n\geq 8$ and $\delta^c(G)\geq \frac{n+1}{2}$,
then $G$ contains two vertex-disjoint rainbow triangles,
and the bound $n\geq 8$ cannot be improved (see Figure \ref{fig3}).
\end{theorem}

\begin{figure}[htbp]
  \centering
 \scalebox{1}{\includegraphics[width=1.5in,height=1.3in]{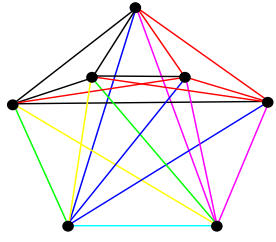}}\\
 \captionsetup{font={scriptsize}}
  \caption{A counterexample for $n=7$ in Theorem \ref{11}}\label{fig3}
\end{figure}

We will also improve the result of Theorem \ref{12} as follows.

\begin{theorem}\label{14}
Let $G$ be an edge-colored graph of order $n$.
If $n\geq 7$ and $\delta^c(G)\geq \frac{n+2}{2}$,
then $G$ contains two vertex-disjoint rainbow triangles, and
the bound $n\geq 7$ cannot be improved (see Figure \ref{fig4}).
\end{theorem}

\begin{figure}[htbp]
  \centering
 \scalebox{1}{\includegraphics[width=1.2in,height=1.2in]{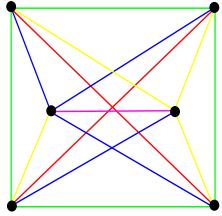}}\\
 \captionsetup{font={scriptsize}}
  \caption{A counterexample for $n=6$ in Theorem \ref{14}}\label{fig4}
\end{figure}

As far as short cycles are concerned in an edge-colored complete graph,
Li et al. showed the following results.

\begin{theorem}[\cite{LBZ}]\label{5}
Let $G$ be an edge-colored complete graph of order $n$.
If $\Delta^{mon}(G)\leq n-2$,
then $G$ contains a properly colored cycle of length at most $4$.
\end{theorem}

\begin{theorem}[\cite{LBZ}]\label{6}
Let $G$ be an edge-colored complete graph of order $n$.
If $\Delta^{mon}(G)\leq n-5$,
then $G$ contains two disjoint properly colored cycle of length at most $4$.
\end{theorem}

In the following sections, we will give the proofs of our Theorems \ref{3}, \ref{8},
\ref{10}, \ref{11} and \ref{14}, separately.

\section{Proof of Theorem \ref{3}}

\textbf{{Proof of Theorem \ref{3}}:}
Let $G$ be a graph satisfying the assumptions of Theorem \ref{3}
and $v$ be a vertex of $G$, and let $t=d^c(v)$.
Suppose $|N_1(v)|=\cdots =|N_s(v)|=1$
and $2\leq |N_{s+1}(v)|\leq \cdots \leq |N_t(v)|$.
Clearly $t-s\leq \frac{n-1-s}{2}$.
Let $N_1=\cup _{1\leq i\leq s}N_i(v)$ and
$N_2=\cup_{s+1\leq i\leq t}N_i(v)$.
Let $R(v)$ be a maximum subset of $E(G)$ such
that for any edge $xy\in R(v)$, $vxyv$ is rainbow.
Then the number of rainbow triangles containing $v$
is equal to $|R(v)|$.
Now we give an orientation to $G[N_1]$ in such a way:
for an edge $xy$, if $c(xy)=c(vx)$,
then the orientation of the edge is from $y$ to $x$,
(if $c(xy)=c(vx)=c(vy)$, then we give the orientation arbitrarily);
if $xy\in R(v)$, then we give the orientation arbitrarily.
The oriented graph is denoted by $D$.
Now we proceed by proving the following claims.

Apparently, all out-arcs from a vertex $u\in N_1$ are assigned colors different from $c(uv)$.
Then we can get the following claim.

\begin{claim}\label{c1}
For all $u\in N_1$, $d^c_{N_1\cup \{v\}}(u)\leq d^+_D(u)+1$.
\end{claim}

\begin{claim}\label{c2}
$C(N_i(v),N_j(v))\setminus R(v)\subseteq \{i,j\}$ for $1\leq i\neq j\leq t$.
\end{claim}

\begin{claim}\label{c3}
$\psi (R(v))\cap N_1\neq \emptyset$.
\end{claim}

\begin{proof}
Suppose not, then all in-arcs to $u\in N_1$ are assigned color $c(uv)$ in $D$.
Then from Claim \ref{c1}, there is a vertex $w$ in $N_1$,
such that $d^c_{N_1\cup \{v\}}(w)\leq \frac{s+1}{2}$.
Since there is no edge in $R(v)$ incident with $w$,
from Claim \ref{c2}, we have $C(w,N_j(v))\subseteq \{c(vw),j\}$.
Hence, $d^c(w)\leq \frac{s+1}{2}+t-s\leq \frac{n}{2}$,
a contradiction.
\end{proof}

Now we proceed the proof of Theorem \ref{3} by induction
on $k$. The case $k=1$ follows from Theorem \ref{1}.
Let $k\geq 2$ and suppose Theorem \ref{3} holds for $k-1$.
Suppose to the contrary, that there is a vertex $v$ such that $|R(v)|<k$.
Since $\delta^c(G)\geq \frac{n+k}{2}\geq \frac{n+k-1}{2}$, we have
$|R(v)|=k-1$.
Certainly, $s\geq k+1$ since $\delta^c(G)\geq \frac{n+k}{2}$.
For $u\in N_1$, let $R_u(v)$ be the subset of $R(v)$ in which each edge is an in-arc to
$u$ and $\psi(R_u(v))\subseteq N_1$
and $R_u'(v)$ be the subset of $R(v)$ in which each edge is incident
with $u$ and $\psi(R_u'(v))\setminus \{u\}\subseteq N_2$,
that is, $R_u(v)=\{uw\in R(v)~|~\overleftarrow{uw}\in D\}$
and $R_u(v)=\{uw\in R(v)~|~w\in N_2\}$.
Hence, from Claims \ref{c1} and \ref{c2},
we have $d^c_{N_1}(u)\leq d^+_{D}(u)+1+|R_u(v)|$ and
$|C(u,N_2)\setminus \{c(uv)\}|\leq t-s+|R_u'(v)|$.
Then,
$$\begin{array}{ll}
d^c(u)&\leq d^+_{D}(u)+1+|R_u(v)|+|R_u'(v)|+t-s.
\end{array}$$
Since $|R_u(v)|+|R_u'(v)|\leq k-1$, if $d^+_{D}(u)\leq \frac{s-k}{2}$,
we have $d^c(u)\leq \frac{n+k-1}{2}$,
a contradiction.
Therefore, $d^+_{D}(u)\geq \frac{s-k+1}{2}$ for $u\in N_1$.
Let $w$ be a vertex with minimum out-degree in $D$.
Then $\frac{s-k+1}{2}\leq d^+_{N_1}(w)\leq \frac{s-1}{2}$.
Assume that $d^+_{D}(w)=\frac{s-k+a}{2}$, $1\leq a\leq k-1$.
Then $|R_w(v)|+|R_w'(v)|\geq k-\frac{a+1}{2}$,
otherwise $d^c(w)<\frac{n+k}{2}$.
Since the edges in $R_w(v)$ are in-arcs to $w$,
they are out-arcs from the vertices in $\psi(R_w(v))\setminus\{w\}$.
Then for $u\in N_1\setminus\{w\}$,
we have $|R_u(v)|+|R_u'(v)|\leq\frac{a-1}{2}$.
So, $d^+_{D}(u)\geq \frac{s+k-a}{2}$ for all $u\in N_1\setminus \{w\}$.
Hence, $\sum_{u\in N_1}d^+_{D}(u)\geq (s-1)\frac{s+k-a}{2}+\frac{s-k+a}{2}$.
Since $s\geq k+1$ and $1\leq a\leq k-1$,
we have $\sum_{u\in N_1}d^+_{D}(u)>\frac{s(s-1)}{2}$,
a contradiction.
$\hfill\qedsymbol$

\section{Proofs of Theorems \ref{8} and \ref{10}}

\textbf{Proof of Theorem \ref{8}:}
Let $G$ be a graph satisfying the assumptions of Theorem \ref{8}
and $v$ be a vertex of $G$.
Suppose, to the contrary, that $G$ has no rainbow triangle.
Assume that $|N_1(v)|=\cdots =|N_s(v)|=1$ and $2\leq |N_{s+1}(v)|\leq \cdots \leq |N_t(v)|$ ($t=d^c(v)$).
Clearly, $t-s\leq \frac{n-1-s}{2}$.
Let $N_1=\cup _{1\leq i\leq s}N_i(v)$ and
$N_2=\cup_{s+1\leq i\leq t}N_i(v)$.
Now we proceed by proving the following claims.

\begin{claim}\label{c4}
$C(N_i(v),N_j(v))\subseteq \{i,j\}$, for $1\leq i\neq j\leq t$.
\end{claim}

\begin{claim}\label{c5}
If $|N_1|\geq 2$, then there is a vertex $u\in N_1$
such that $C(u,N_1\setminus \{u\})=\{c(vu)\}$.
\end{claim}

\begin{proof}
Suppose not, let $u$ be a vertex in $N_1$
with minimum color-degree in $G[N_1]$.
Set $W_1=\{w\in N_1~|~c(uw)=c(vu)\}$
and $W_2=\{w\in N_1~|~c(uw)= c(vw)\}$.
From Claim \ref{c4}
we have $N_1\setminus\{u\}=W_1\cup W_2$.
For any vertex $w_1\in W_1$ and $w_2\in W_2$,
from Claim \ref{c4},
we have $c(w_1w_2)\in \{c(vw_1),c(vw_2)\}$.
According to the definitions of $N_1$, $W_1$ and $W_2$,
we have $c(uw_1)\neq c(uw_2)$.
Then, $c(w_1w_2)=c(uw_2)$,
otherwise $uw_1w_2u$ is a rainbow triangle, a contradiction.
Hence, for any vertex $w\in W_2$,
$C(w,W_1)=\{c(vw)\}$.
Then, there is a vertex $w$ in $W_2$ such that
$d^c_{N_1}(w)<d^c_{N_1}(u)$.
Therefore, we can find a vertex $u\in N_1$
such that $C(u,N_1\setminus \{u\})=\{c(vu)\}$.
\end{proof}

Next we distinguish two cases.

\textbf{Case 1.} $|N_1|\geq 2$.

Let $u$ be a vertex in $N_1$ such that $C(u,N_1\setminus \{u\})=\{c(vu)\}$.
From Claim \ref{c4}, we have $C(u,N_j(v))\subseteq \{c(vu),j\}$ for $s+1\leq j\leq t$.
Hence, $d^c(u)\leq t-s+1\leq \frac{n-1-s}{2}+1\leq \frac{n-1}{2}$,
a contradiction.

\textbf{Case 2.} $|N_1|=1$.

Let $N_1=\{u\}$. Then $|N_j(v)|=2$ and $j\in C(u,N_j(v))$ for $2\leq j\leq t$,
otherwise $d^c(u)<\frac{n}{2}$.
Assume there is a set $N_k(v)=\{x,y\}$.
W.l.o.g., suppose $c(ux)=c(vx)=k$.
Since $n\geq 8$, there exists a vertex $z\in N_2\setminus N_k(v)$
such that $c(xz)=c(vz)$.
Hence, $c(uz)=c(vz)$,
otherwise from Claim \ref{c4},
$uxzu$ is a rainbow triangle.
Therefore,
$c(zy)=c(vy)=k$,
otherwise $d^c(z)\leq t-s-2+2\leq \frac{n}{2}-1$, a contradiction.
Hence, $c(zx)\neq c(zy)$.
Then, $c(xy)\in\{c(zx),c(zy)\}= \{c(vx),c(vz)\}$,
otherwise $xyzx$ is a rainbow triangle.
So, $d^c(x)\leq t-s-1+1\leq \frac{n}{2}-1$, a contradiction.

$\hfill\qedsymbol$

\textbf{{Proof of Theorem \ref{10}}:}
Let $G$ be a graph satisfying the assumptions of Theorem \ref{10}.
Since $G$ has no rainbow triangle and $\delta^c(G)\geq \frac{n-1}{2}$,
there exist a vertex $v$ such that $d^{c}(v)=\frac{n-1}{2}=t$ in $G$.
Assume that $|N_1(v)|=\cdots =|N_s(v)|=1$
and $2\leq |N_{s+1}(v)|\leq \cdots \leq |N_t(v)|$.
Let $N_1=\cup _{1\leq i\leq s}N_i(v)$ and
$N_2=\cup_{s+1\leq i\leq t}N_i(v)$.
Now we proceed by proving the following claims.

\begin{claim}\label{c7}
$C(N_i(v),N_j(v))\subseteq \{i,j\}$, for $1\leq i\neq j\leq t$.
\end{claim}

\begin{claim}\label{c8}
$N_1=\emptyset$.
\end{claim}

\begin{proof}
Suppose not, since $d^c(v)=\frac{n-1}{2}$,
there is a set $N_k(v)$, $s+1\leq k\leq t$,
such that $|N_k(v)|\geq 3$.
Thus, $s-t\leq \frac{n-s-2}{2}$.
From Claim \ref{c7}, $C(u,N_j(v))\subseteq \{c(vu),j\}$ for $s+1\leq j\leq t$.

If $|N_1|\geq 2$,
as in the proof of Theorem \ref{8},
there is a vertex $u\in N_1$
such that $C(u,N_1\setminus \{u\})=\{c(vu)\}$.
Then, $d^c(u)\leq t-s+1\leq \frac{n-2-s}{2}+1\leq \frac{n}{2}-1$,
a contradiction.
If $|N_1|=1$, let $N_1=\{u\}$, and assume that $|N_2(v)|=3$
and $|N_j(v)|=2$ for $3\leq j\leq t$.
Then there must exist a vertex $x\in N_l(v)$
such that $c(ux)=c(vx)$, $3\leq l\leq t$.
Let $N_l(v)=\{x,y\}$.
Then there is a vertex $z\in N_k(v)$ such that $c(xz)=c(vz)$,
otherwise $d^c(x)<\frac{n-1}{2}$.
Hence, $c(zy)=c(vy)$,
otherwise $d^c(z)<\frac{n-1}{2}$.
Then $c(xz)\neq c(yz)$.
Therefore, $c(xy)\in \{c(vz),c(vx)\}$,
otherwise $xyzx$ is a rainbow triangle.
Then, $d^c(x)\leq \frac{n-2-1}{2}<\frac{n-1}{2}$,
a contradiction.
\end{proof}

By Claim \ref{c8},
$|N_i(v)|=2$ for all $1\leq i\leq t$.
Thus, $n$ is odd.
Since $\delta^c(G)\geq \frac{n-1}{2}$,
by Claim \ref{c7},
we have that for any vertex $u\in N_i(v)$,
$j\in C(u,N_j(v))$ for $1\leq j\neq i\leq t$,
that is $C(u, N_j(v))=\{c(vu),j\}$.
If $n\leq 5$,
it is easy to verify that
$C(N_i(v))\subset \{1,2\}$.
If $n\geq 7$,
then $C(N_j(v))=\{j\}$.
Otherwise, suppose that there is a set $N_k(v)=\{x,y\}$
such that $c(xy)\neq k$.
Then there is a vertex $z\in N_2\setminus N_k(v)$ such that
$xyzx$ is a rainbow triangle, a contradiction.
Therefore, let $A_0=\{v\}$ and $A_i=N_i(v)$ for $1\leq i\leq \frac{n-1}{2}$.
This completes the proof.
$\hfill\qedsymbol$

\section{Proofs of Theorems \ref{11} and \ref{14}}

At first we need the following lemmas.

\begin{lemma}\label{l1}
Let $G$ be an edge-colored complete graph of order $n\geq 8$.
If $\delta^c(G)\geq \frac{n+1}{2}$ and
there are two vertices $y,z$ such that $G'=G-\{y,z\}$
has no rainbow triangles,
then $G$ has two vertex-disjoint rainbow triangles containing $y$ and $z$,
respectively.
\end{lemma}

\begin{proof}
Since $\delta^c(G)\geq \frac{n+1}{2}$,
we have $\delta^c(G')\geq \delta^c(G)-2\geq \frac{|G'|-1}{2}$.
From Theorem \ref{10},
$d^c_{G'}(v)=\frac{|G'|-1}{2}$ for $v\in V(G')$ and
$G'$ has a partition $\{A_i,0\leq i\leq \frac{n-1}{2} \}$.
Assume that $A_0=\{v\}$ and $A_i=N_i(v)$ for $1\leq i\leq \frac{n-1}{2}$.
Since $\delta^c(G)\geq \frac{n+1}{2}$,
the edges from every vertex in $G'$ to $z$ and $y$
are assigned two new colors.
Let $N_i(v)=\{a_i,b_i\}$.
If there is a set $N_k(v)$ such that $c(za_k)\neq c(zb_k)$,
then $C(y,G'\setminus (\{v\}\cup N_k(v)))=\{c(vy)\}$.
If not, there is a vertex $u\in G'\setminus (\{v\}\cup N_k(v))$
such that $c(uy)\neq c(vy)$, and then $uvyu$ is a disjoint rainbow triangle from $xa_kb_kx$.
So, $d^c(y)\leq 4$, a contradiction.
Hence, for any set $N_i(v)$, $c(za_i)=c(zb_i)$ and $c(ya_i)=c(yb_i)$.
Since $d^c(z)\geq \frac{n+1}{2}$,
there are two sets $N_i(v)$ and $N_j(v)$, $i\neq j$,
such that $c(za_i)\neq c(za_j)$.
Then $za_ia_jz$ is a rainbow triangle.
Similarly, we can find another rainbow triangle $yb_kb_ly$,
a contradiction.
\end{proof}

\begin{lemma}\label{l2}
Let $G$ be an edge-colored graph of order $n\geq 7$.
If $\delta^c(G)\geq \frac{n+2}{2}$ and
there are two vertices $y,z$ such that $G'=G-\{y,z\}$
has no rainbow triangles,
then $G$ has two vertex-disjoint rainbow triangles containing $y$ and $z$,
respectively.
\end{lemma}

\begin{proof}
Since $\delta^c(G)\geq \frac{n+1}{2}$,
we have $\delta^c(G')\geq \delta^c(G)-2\geq \frac{|G'|-1}{2}$.
According to Theorem \ref{16},
we know that $G'$ is a properly colored balanced complete bipartite graph.
Since $\delta ^c(G)\geq \frac{n+2}{2}$,
we have that $vz$ and $zy$ are in $E(G)$ and $d^c(v)=\frac{n+2}{2}$ for every vertex $v\in V(G)$.
Thus, we can easily find two vertex-disjoint rainbow triangles containing $z$ and $y$, respectively,
a contradiction.
\end{proof}

\textbf{Proof of Theorem \ref{11}:}
Since $\delta^c(G)\geq \frac{n}{2}$,
according to Theorem \ref{8},
$G$ has a rainbow triangle $xyzx$.
Let $T(G)$ be the set of all rainbow triangles in $G$.
Suppose, to the contrary,
that $G$ has no vertex-disjoint rainbow triangles.
Then each rainbow triangle in $T(G)\setminus \{xyzx\}$
meets at least one of $\{x,y,z\}$.
Let $W_1 (W_2,W_3)$ denote the subset of vertices in $V(G)\setminus \{x,y,z\}$,
in which every vertex is contained in a rainbow triangle together with $x(y,z)$.
Now we proceed by proving the following claims.

\begin{claim}\label{c12}
$W_i\neq \emptyset$, for $i=1,2,3$.
\end{claim}

\begin{proof}
W.l.o.g., suppose $W_1=\emptyset$.
Then each rainbow triangle in $T(G)\setminus \{xyzx\}$
meets $y$ or $z$. From Lemma \ref{l1},
we can find two vertex-disjoint rainbow triangles in $G$,
a contradiction.
\end{proof}

\begin{claim}\label{c10}
For any set $W_i$,
there is a vertex in $W_i$ but not in $W_j$, $1\leq i\neq j \leq 3$.
\end{claim}

\begin{proof}
W.l.o.g., suppose, to the contrary,
that $W_1\subseteq W_2\cup W_3$.
Then there is no rainbow triangle in $G'=G-\{z,y\}$.
Therefore, from Lemma \ref{l1}
we can find two vertex-disjoint rainbow triangles in $G$,
a contradiction.
\end{proof}

\begin{claim}\label{c11}
There is a vertex $a_0$ such that
all rainbow triangles in $T(G)\setminus \{xyzx\}$
meet at $a_0$.
\end{claim}

\begin{proof}
By Claim \ref{c10},
let $a_i\in W_i\setminus (W_j\cup W_k)$, $1\leq i\neq j\neq k\leq 3$.
Then there is a vertex $a_0$ such that
$xa_1a_0x$, $ya_2a_0y$ and $za_3a_0z$ are rainbow triangles,
otherwise we can easily find two disjoint rainbow triangles.
Suppose that there is a rainbow triangle $uvwu$ in $T(G)\setminus \{xyzx\}$
such that $a_0\notin\{u,v,w\}$.
Then we can easily find two vertex-disjoint rainbow triangles, a contradiction.
\end{proof}

Let $G''=G-\{x,a_0\}$.
Then $\delta^c(G'')\geq \frac{|G''|-1}{2}$.
From Claim \ref{c11},
there is no rainbow triangle in $G''$.
Hence, from Lemma \ref{l1}
we can find two vertex-disjoint rainbow triangles in $G$ containing $x$ and $a_0$,
respectively,
a contradiction.
This completes the proof of Theorem \ref{11}.
$\hfill\qedsymbol$\\

\textbf{Proof of Theorem \ref{14}:}
Using Theorems \ref{15} and \ref{16} in \cite{LNXZ}, and
Lemma \ref{l2} as well as by an analogue of the proof of Theorem \ref{11},
we can get the result of Theorem \ref{14}.
$\hfill\qedsymbol$


\begin{thebibliography}{10}

\bibitem{B} J.A. Bondy, U.S.R. Murty, Graph Theory, GTM {\bf 244}, Springer (2008).

\bibitem{FM} S. Fujita, C. Magnant,
Properly colored paths and cycles,
Discrete Appl. Math. {\bf 159} (2011), 1391--1397.

\bibitem{HLY} J. Hu, H. Li, D. Yang,
Vertex-disjoint rainbow triangles in edge-colored graphs,
Discrete Math. {\bf 343} (2020), 112--117.

\bibitem {L} H. Li, Rainbow $C_3$'s and $C_4$'s in edge-colored graphs,
Discrete Math. {\bf 313} (2013) 1893--1896.

\bibitem{LBZ} R. Li, H. Broersma, S. Zhang,
Vertex-disjoint properly edge-colored cycles in edge-colored complete graphs,
J. Graph Theory {\bf 94} (2020), 476--493.

\bibitem{LNXZ} B. Li, B. Ning, C. Xu, S. Zhang,
Rainbow triangles in edge-colored graphs,
European J. Combin. {\bf 36} (2014), 453--459.
\end{thebibliography}
\end{document}